\documentclass[a4paper]{amsart}
\usepackage{amsmath,amsthm,amssymb,hyperref,mathrsfs}
\usepackage{aliascnt}
\usepackage{lmodern}
\usepackage[T1]{fontenc}

\usepackage[textsize=footnotesize]{todonotes}

\usepackage{xcolor}
\definecolor{dblue}{rgb}{0,0,0.70}
\hypersetup{
	unicode=true,
	colorlinks=true,
	citecolor=dblue,
	linkcolor=dblue,
	anchorcolor=dblue
}

\makeatletter
\expandafter\g@addto@macro\csname th@plain\endcsname{%
		\thm@notefont{\bfseries}
	}%
\expandafter\g@addto@macro\csname th@remark\endcsname{%
		\thm@headfont{\bfseries}
	}%
\makeatother

\newtheorem{theorem}
{Theorem}[section]	
\newtheorem*{theorem*}{Theorem}

\newaliascnt{lemma}{theorem}
\newtheorem{lemma}[lemma]{Lemma}

\aliascntresetthe{lemma}
\newtheorem*{lemma*}{Lemma}

\newaliascnt{proposition}{theorem}

\aliascntresetthe{proposition}

\newaliascnt{corollary}{theorem}
\newtheorem{corollary}[corollary]{Corollary}
\aliascntresetthe{corollary}

\theoremstyle{remark}

\newaliascnt{remark}{theorem}
\newtheorem{remark}[remark]{Remark}
\aliascntresetthe{remark}
\newaliascnt{question}{theorem}
\newtheorem{question}[question]{Question}
\aliascntresetthe{question}

\newtheorem*{question*}{Question}

\newaliascnt{definition}{theorem}
\newtheorem{definition}[definition]{Definition}
\aliascntresetthe{definition}

\newaliascnt{example}{theorem}

\aliascntresetthe{example}

\renewcommand{\restriction}{\mathbin\upharpoonright}

\newcommand{\axiom}[1]{\mathsf{#1}} 
\newcommand{\ZFC}{\axiom{ZFC}}
\newcommand{\AC}{\axiom{AC}}
\newcommand{\CH}{\axiom{CH}}

\newcommand{\DC}{\axiom{DC}}
\newcommand{\ZF}{\axiom{ZF}}

\newcommand{\GCH}{\axiom{GCH}}

\newcommand{\HS}{\axiom{HS}}

\newcommand{\SVC}{\axiom{SVC}}

\DeclareMathOperator{\sym}{sym}

\DeclareMathOperator{\fix}{fix}

\DeclareMathOperator{\id}{id}
\DeclareMathOperator{\aut}{Aut}
\DeclareMathOperator{\Col}{Col}
\DeclareMathOperator{\Add}{Add}

\newcommand{\forces}{\mathrel{\Vdash}}

\newcommand{\incompatible}{\mathrel{\bot}}
\newcommand\PP{\mathbb{P}}

\newcommand{\QQ}{\mathbb{Q}}
\newcommand{\RR}{\mathbb{R}}

\newcommand{\sF}{\mathscr F}

\newcommand{\sG}{\mathscr G}

\newcommand{\cL}{\mathcal L}

\newcommand{\tup}[1]{\langle#1\rangle}

\author{Asaf Karagila}
\thanks{The author was supported by the Royal Society grant no.~NF170989.}
\email[Asaf Karagila]{karagila@math.huji.ac.il}
\urladdr{http://karagila.org}
\address{School of Mathematics,
University of East Anglia.
Norwich, NR4~7TJ, UK
}
\date{\today}

\title{Preserving Dependent Choice}

\begin{document}
\begin{abstract}
We prove some general theorems for preserving Dependent Choice when taking symmetric extensions, some of which are unwritten folklore results. We apply these to various constructions to obtain various simple consistency proofs.
\end{abstract}
\maketitle
\section{Introduction}
Dependent Choice is one of the best known weak versions of the axiom of choice, and perhaps the most natural version of the axiom of choice. Indeed, Dependent Choice---or $\DC$---is sometimes mistaken as countable choice, and while it is strong enough to provide us with the basis of analysis (Baire Category Theorem, well-behaved theory of Borel sets and measure, etc.), it is also consistent with assumptions such as ``all sets of reals are regular'' for many versions of regularity (e.g., Lebsegue measurability). 

Therefore, in many constructions of models without the axiom of choice, it is often desirable to preserve $\DC$. We sometimes have to work quite hard for that, and sometimes it is quite easy to obtain. The purpose of this note is to provide some straightforward conditions which allow for the preservation of $\DC$, as well as its stronger versions $\DC_{<\kappa}$ for some infinite cardinal $\kappa$.

The arguments shown here should be considered folklore, even if no explicit formulation or proof appeared in print at this level of generality until today.\footnote{With the exception of the author putting into print the folklore \autoref{lemma:closure} in \cite{Karagila:2014}.} They were used by different people over the years, even if applied to specific constructions each time.

\subsection*{Acknowledgements}
The author would like to thank Amitayu Banerjee for pointing out the simple solution appearing in \autoref{app:apter}, as well as to Philipp Schlicht for asking him the question that led to \autoref{lemma:proper}.

\section{Preliminaries}
Our notation is mostly standard, and follows Jech for the most part. We use $\DC_\lambda$ to denote the statement ``Every $\lambda$-closed tree has a chain of length $\lambda$ or a maximal element'', for the case $\lambda=\omega$ we just write $\DC$, and if $\lambda$ is a limit cardinal, $\DC_{<\lambda}$ is the abbreviation of $\forall\kappa<\lambda, \DC_\kappa$. Similarly, $\AC_\lambda$ will denote the statement ``Every family of $\lambda$ non-empty sets admits a choice function'' and $\AC_{<\lambda}$ abbreviates $\forall\kappa<\lambda,\AC_\kappa$ (we will not use $\AC$ to mean $\AC_\omega$, though, as $\AC$ denotes the axiom of choice).

Much has been written on $\DC$, for example, if $\lambda$ is singular then $\DC_{<\lambda}$ implies $\DC_\lambda$, see \cite[Chapter~8]{Jech:AC1973} for details and more.
\begin{theorem}[Folklore]
For a regular cardinal $\lambda$, $\DC_\lambda$ holds if and only if every $\lambda^+$-closed forcing is $\lambda^+$-distributive.
\end{theorem}
\begin{proof}[Sketch of Proof]
Assuming $\DC_\lambda$ the standard proof in $\ZFC$ translates immediately. In the other direction, if $\DC_\lambda$ fails, there is a tree which is $\lambda$-closed, but has no $\lambda$-chains and thus it is vacuously $\lambda^+$-closed, however forcing with the tree adds a $\lambda$-sequence, so it is a witness to the failure of the distributivity.
\end{proof}

We will say that a class $A$ is \textit{$\kappa$-closed} if $A^{<\kappa}\subseteq A$. Of course this depends on the universe, and when that is not clear from context we will be sure to explicitly state what is the universe to which the closure is relative.

\subsection{Some Forcing Shorthands}

We will want to define and manipulate names in a fairly explicit manner. This means that we cannot make the usual simplifying assumptions that let us choose arbitrary name with this and that kind of properties. To that end, we define a few shorthand notations.

We say that a name $\dot y$ \textit{appears} in a name $\dot x$, if there is an ordered pair $\tup{p,\dot y}\in\dot x$. We similarly say a condition $p$ appears in $\dot x$ if there is an ordered pair $\tup{p,\dot y}\in\dot x$.

If $\{\dot y_i\mid i\in I\}$ is a collection of names, we define $\{\dot y_i\mid i\in I\}^\bullet$ to be the obvious way of turning it into a name, namely, $\{\tup{1,\dot y_i}\mid i\in I\}$. This extends to other very canonical definitions, e.g.\ $\tup{\dot x,\dot y}^\bullet$ is the simplest way of creating the name of an ordered pair with $\dot x$ and $\dot y$. Using this notation, by the way, $\check x=\{\check y\mid y\in x\}^\bullet$.

Additionally, if $\dot x$ is a name, and $p$ is a condition, we write $\dot x\restriction p$ for the name $\{\tup{q,\dot y}\mid q\leq p, \dot y\text{ appears in }\dot x, q\forces\dot y\in\dot x\}$. It is easy to verify that $p\forces\dot x=\dot x\restriction p$, and if $q\incompatible p$, then $q\forces\dot x\restriction p=\varnothing$.

\subsection{Symmetric Extensions}
If $\PP$ is a forcing, and $\pi$ is an automorphism of $\PP$, then $\pi$ extends to $\PP$ names recursively:\[\pi\dot x=\{\tup{\pi p,\pi\dot y}\mid\tup{p,\dot y}\in\dot x\}.\]
This action also respects the forcing relation, as shown in the following lemma.
\begin{lemma*}[The Symmetry Lemma]
Suppose that $p\in\PP$ is a condition, $\dot x$ is a $\PP$-name, $\varphi(\dot x)$ is a formula in the language of forcing, and $\pi$ is an automorphism of $\PP$, then\[p\forces\varphi(\dot x)\iff\pi p\forces\varphi(\pi\dot x).\qed\]
\end{lemma*}

Fix an automorphism group $\sG\leq\aut(\PP)$. We say that $\sF$ is a \textit{normal filter of subgroups} of $\sG$ if it is a filter on the lattice of subgroups which is closed under conjugations. Namely, $\sF$ is closed under supergroups (with respect to $\sG$) and intersections, and if $\pi\in\sG$ and $H\in\sF$, then $\pi H\pi^{-1}\in\sF$ as well. If $\PP$ is a forcing, $\sG$ is an automorphism group of $\PP$, and $\sF$ is a normal filter of subgroups of $\sG$ we say that $\tup{\PP,\sG,\sF}$ is a \textit{symmetric system}.\footnote{In many cases, it is enough to consider a normal filter base of subgroups, rather than a filter of subgroups, and we will not bother to make the distinction.}

We say that a $\PP$-name is \textit{$\sF$-symmetric} if $\sym_\sG(\dot x)=\{\pi\in\sG\mid\pi\dot x=\dot x\}\in\sF$, and if this property holds hereditarily for names which appear in $\dot x$, we say that $\dot x$ is \textit{hereditarily $\sF$-symmetric}. The class of hereditarily $\sF$-symmetric names is denoted by $\HS_\sF$. We will omit the subscripts when it is clear what is the symmetric system, which will be most of the time.

\begin{theorem*}
Suppose that $\tup{\PP,\sG,\sF}$ is a symmetric system, let $G$ be a $V$-generic filter, then $M=\HS^G=\{\dot x^G\mid\dot x\in\HS\}$ is a transitive class of $V[G]$ such that $V\subseteq M\subseteq V[G]$, and $M$ is a class model of $\ZF$. 
\end{theorem*}
The model $M$ in the above theorem is called a \textit{symmetric extension} of $V$. Finally, we have a symmetric forcing relation, namely $\forces^\HS$, which is the relativization of the forcing to the symmetric extension, which satisfies the same basic properties as $\forces$.

\begin{definition}
If $\tup{\PP,\sG,\sF}$ is a symmetric system, we say that a condition $p\in\PP$ is \textit{tenacious}, if $\{\pi\in\sG\mid\pi p=p\}\in\sF$. We say that $\PP$ is tenacious, if there is a dense subset of tenacious conditions.
\end{definition}
It turns out that this concept is somehow a bit redundant, and if $\tup{\PP,\sG,\sF}$ is a symmetric system, then we can define a forcing $\PP^*\subseteq\PP$ such that $\tup{\PP^*,\sG,\sF}$ is equivalent to $\tup{\PP,\sG,\sF}$ and it is is not only tenacious, but in fact every condition is tenacious, as was shown in \cite[\S12]{Karagila:2016}. So when it is useful, we may assume $\PP$ is tenacious without loss of generality.
\section{Dependent Choice in Symmetric Extensions}
In \cite{Karagila:2014} we proved the following folklore lemma.
\begin{lemma}[Lemma~2.1 in \cite{Karagila:2014}]\label{lemma:closure}
Assume $\ZFC$ holds. Suppose that $\tup{\PP,\sG,\sF}$ is a symmetric system such that $\PP$ is $\lambda$-closed and $\sF$ is $\lambda$-complete, then $\forces^\HS\DC_{<\lambda}$.
\end{lemma}
The proof is simple enough to merit repeating. But to simplify generalizations and variations, we will first extract the following lemma from the standard proof.

\begin{lemma}\label{lemma:model-closure}
Suppose that $M$ is a $\lambda$-closed inner model of $N$. If $N\models\DC_{<\lambda}$, then $M\models\DC_{<\lambda}$.
\end{lemma}
\begin{proof}
Suppose that $N$ satisfies $\DC_{<\lambda}$, and let $T\in M$ be a $\kappa$-closed tree without maximal element, for some $\kappa<\lambda$. By $\lambda$-closure of $M$, $T$ is also $\kappa$-closed in $N$, and therefore has a branch there, and this branch is a function from $\kappa$ to $T$, so it is in $M$, as wanted.
\end{proof}

\begin{proof}[Proof of \autoref{lemma:closure}]
Let $G$ be a $V$-generic filter, and let $M$ be $\HS^G$. It is enough, by the previous lemma, to prove that $M^\kappa\subseteq M$ for all $\kappa<\lambda$. And indeed, if $f\colon\kappa\to M$ for some $\kappa<\lambda$, let $\dot f$ be a name for $f$ such that all the names appearing in $\dot f$ are of the form $\tup{\check\alpha,\dot y}^\bullet$ where $\dot y\in\HS$.

Let $p$ be any condition such that $p\forces``\dot f$ is a function'', set $p_0=p$, and recursively extend $p_\alpha$ to $p_{\alpha+1}$ such that $p_{\alpha+1}$ decides the value of $\dot f(\check\alpha)$, going through limit steps using the fact that $\PP$ is $\lambda$-closed. Finally, for all $\alpha<\kappa$ there is a name $\dot y_\alpha\in\HS$ such that $p_\kappa\forces\dot f(\check\alpha)=\dot y_\alpha$, define $\dot g=\{\tup{\check\alpha,\dot y_\alpha}^\bullet\mid\alpha<\kappa\}^\bullet$. 

Let $H=\bigcap_{\alpha<\kappa}\sym(\dot y_\alpha)$. By $\lambda$-closure of $\sF$, $H\in\sF$. It is easy to verify that $H$ is a subgroup of $\sym(\dot g)$, so $\dot g\in\HS$ and $p_\kappa\forces\dot g=\dot f$. This means that there is a dense open set of conditions $q\leq p$ such that for some $\dot g\in\HS$, $q\forces\dot g=\dot f$, so by genericity, $\dot f^G=f\in M$ as wanted. Now by the previous lemma, $M\models\DC_\kappa$ for all $\kappa<\lambda$.
\end{proof}
\begin{lemma}\label{lemma:chain-condition}
We can replace ``$\PP$ is $\lambda$-closed'' by ``$\PP$ has the $\lambda$-c.c.'' in \autoref{lemma:closure}.\footnote{Amitayu Banerjee had let us know that he had independently made a similar observation.}
\end{lemma}
\begin{proof}[Sketch of Proof]
We again appeal to the argument that $M$ is $\lambda$-closed in $V[G]$. Suppose that $\dot f$ is a $\PP$-name for a function $f\colon\kappa\to M$. 

For every $\alpha<\kappa$, let $D_\alpha$ be a maximal antichain of conditions $p$ such that for some $\dot y_p\in\HS$, $p\forces\dot f(\check\alpha)=\dot y_p$. We can now define $\dot y_\alpha$ to be the name obtained by $\bigcup_{p\in D_\alpha}\dot y_p\restriction p$. Without loss of generality we may assume that each condition is tenacious, so by intersecting, we can assume that $\pi\in\sym(\dot y_p)$ means that $\pi p=p$. In particular, by $\lambda$-completeness of $\sF$, $\bigcap_{p\in D_\alpha}\sym(\dot y_p)\in\sF$ and it is easy to see that this is a subgroup of $\sym(\dot y_\alpha)$. Therefore, $\dot y_\alpha\in\HS$.

It follows that $H=\bigcap_{\alpha<\kappa}\sym(\dot y_\alpha)$ is in $\sF$ and therefore $\dot g=\{\tup{\check\alpha,\dot y_\alpha}^\bullet\mid\alpha<\kappa\}^\bullet$ is such that $\dot g\in\HS$. Since $\forces\dot f=\dot g$, it follows that $f\in M$.
\end{proof}
This shows that if $\sF$ is $\sigma$-closed, both c.c.c.\ and $\sigma$-closed forcings would preserve $\DC$. Philipp Schlicht raised a natural question, will properness suffice?
\begin{lemma}\label{lemma:proper}
If $\PP$ is proper and $\sF$ is $\sigma$-complete, then $\DC$ is preserved.
\end{lemma}
\begin{proof}[Sketch of Proof]
Let $G$ be a $V$-generic filter, and $M$ the symmetric extension given by $\HS^G$. Suppose that $f\colon\omega\to M$ is a function, and let $\dot f$ be a name for it, and some $p$ which forces that $\dot f$ is a function into $\HS$. 

Let $N$ be a countable elementary submodel of $H(\theta)$ for a sufficiently large $\theta$, with $\tup{\PP,\sG,\sF},\dot f,p\in N$. By elementarity, $N$ contains ``enough'' names from $\HS$ to compute all the possible values of $\dot f$, as there are only countably many of those, we can intersect the relevant groups and remain in $\sF$ to fix all the necessary names. Next, find an $N$-generic condition extending $p$, and use it to define a name for a function in $\HS$ which the $N$-generic condition will force to be equal to $f$. By density, this must have happened in $V[G]$, so $f\in M$.
\end{proof}

The keen eyed reader might have noticed at this point that all these proofs are the same flavor: the symmetric extension is closed under ${<}\kappa$-sequences in the full extension. Does that provide us with a full characterization of symmetric extensions which satisfy $\DC_{<\kappa}$? 

The answer is negative, as to be expected. Let $\tup{\PP,\sG,\sF}$ be any symmetric system which preserve $\DC_{<\kappa}$, by consider the product of $\tup{\PP,\sG,\sF}$ with the symmetric system $\tup{\Add(\omega,1),\aut(\Add(\omega,1)),\{\aut(\Add(\omega,1))\}}$. Namely, we take the product of $\PP$ with adding a single Cohen real, the full automorphism group, and the trivial filter of subgroups. It is not hard to see that only $\PP$-names can be symmetric in this extension, so the symmetric extension is the same as that given just by $\tup{\PP,\sG,\sF}$, but the full generic extension contains a Cohen real, therefore $\sigma$-closure is violated.

But is this the only trivial obstruction? The following theorem shows that morally, the answer is yes. We will need the axiom $\SVC$, or ``Small Violation of Choice'' formulated by Andreas Blass in \cite{Blass:1979}. The axiom can be stated as ``The axiom of choice can be forced with a set-forcing''. In particular, symmetric extensions satisfy $\SVC$, at least under the assumption that the ground model did.
\begin{theorem}
Suppose that $M\models\DC_{<\kappa}+\SVC$, then $M$ is $\kappa$-closed in a model of $\ZFC$.
\end{theorem}
\begin{proof}
Without loss of generality we can assume that $\kappa$ is the least such that $\DC_\kappa$ fails. From \cite[Theorem~8.1]{Jech:AC1973} it follows that $\kappa$ is regular. Recall that $\SVC$ can be restated as ``there exists a set $X$ such that forcing a well-ordering of $X$ forces the axiom of choice''. Since $\DC_{<\kappa}$ holds, we can force a well-ordering of $X$ of type $\kappa$ by initial segments. By $\DC_{<\kappa}$ this forcing is $\kappa$-closed and does not add ${<}\kappa$-sequences. Therefore $M$ is a $\kappa$-closed inner model of a model of $\ZFC$.
\end{proof}

$\SVC$ should not be necessary, but it is \textit{somewhat} necessary. On the one hand it is easy to construct a class-symmetric extension which is $\kappa$-closed, but does not satisfy $\SVC$ (e.g.\ the class extensions given in \cite{Karagila:2014}). On the other hand, if it is consistent (modulo large cardinal hypotheses) with $\ZF+\DC$ that all successor cardinals have cofinality $\omega_1$, or in a generalized Morris-style model satisfying $\DC$ (see \cite{Karagila:2018} for details),\footnote{Neither statements are known to be consistent with $\ZF+\DC$. We conjecture the latter is consistent.} then such a model cannot be extended to a model of $\ZFC$ without adding ordinals. In particular, this model is not $\aleph_1$-closed in a model of $\ZFC$.

Finally, we remark that it is quite easy to verify that a $\sigma$-closed forcing must preserve $\DC$. In a more general way, we can prove that a proper forcing cannot violate $\DC$. For a complete discussion on the topic of properness in $\ZF$, see the author's work with David Asper\'o in \cite{AsperoKaragila:2018}.
\section{Some Applications}
\subsection{Failures of \texorpdfstring{$\GCH$}{GCH} at limit cardinals below a supercompact cardinal}\label{app:apter}
Arthur Apter proved in \cite{Apter:2012} the following theorem:

\begin{theorem*}[Apter, Theorem~3]
Assume $V\models\ZFC+\GCH+\kappa$ is supercompact. Then there is a symmetric extension in which $\AC_\omega$ fails, $\kappa$ is a regular limit cardinal and supercompact, and $\GCH$ holds at a limit cardinal $\delta$ if and only if $\delta>\kappa$.
\end{theorem*}

Of course, there are some concessions to be made. Supercompactness here is meant in the sense of ultrafilters, which is weaker than the sense of embedding (e.g.\ $\omega_1$ can be supercompact in the sense used by Apter, but it cannot be the critical point of an elementary embedding). In addition $\GCH$ is weakened to mean that there is no injection from $\delta^{++}$ into $\mathcal P(\delta)$, this is because of the classical theorem that $\GCH$ (in its standard formulations) implies the axiom of choice.

At the end Apter asks whether or not this result can be improved by having some weak form of the axiom of choice hold. Amitayu Banerjee pointed out that \autoref{lemma:chain-condition} gives a simple answer based on Apter's original construction.

\begin{theorem}
Assume $V\models\ZFC+\GCH+\kappa$ is supercompact. Then there is a symmetric extension in which $\DC_{<\kappa}$ holds, $\kappa$ is a regular limit cardinal and supercompact, and $\GCH$ holds for a limit cardinal $\delta$ if and only if $\delta>\kappa$.
\end{theorem}
\begin{proof}
Apter's proof begins by preparing $V$ so that $\kappa$ is indestructibly supercompact and that there is a club $C\subseteq\kappa$ such that $\min C=\omega$ and the successor points are inaccessible, such that for all $\delta\in C$, $2^\delta=2^{\delta^+}=\delta^{++}$.

Let $\tup{\kappa_i\mid i<\kappa}$ be a continuous enumeration of $C$, then $\PP$ is the Easton support product of $\Col(\kappa_i^{++},{<}\kappa_{i+1})$. We take $\sG$ to be the Easton support product of the automorphism groups of each collapse, and $\sF$ is the filter generated by the groups of the form $\fix(\alpha)=\{\pi\in\prod_{i\in C}\aut(\Col(\kappa_i^{++},{<}\kappa_{i+1}))\mid\pi\restriction\alpha=\id\}$. Namely, the filter is generated by groups which concentrate on only applying permutations above some fixed initial segment. Let $G$ be a $V$-generic filter for $\PP$ and let $M$ denote the symmetric extension.

Easily, $\sF$ is $\kappa$-complete, and the Easton product is $\kappa$-c.c., so by \autoref{lemma:chain-condition} $\DC_{<\kappa}$ holds, and by Lemma~3.3 in \cite{Hayut-Karagila:2018} $\kappa$ remains supercompact. Since $\GCH$ held in $V$ above $\kappa$, and $\PP\subseteq V_\kappa$, it follows that for any limit cardinal $\delta>\kappa$, there is no injection from $\delta^{++}$ into $\mathcal P(\delta)$, since there is no such injection in $V[G]$, which agree with $V$ on cardinals above $\kappa$.

It remains to show that if $\delta\leq\kappa$ is a limit cardinal, then $\delta^{++}$ can be injected into $\mathcal P(\delta)$. For this note that $V_\kappa^M=V_\kappa^{V[G]}$, so it is enough to prove this in $V[G]$.

First, note that if $\delta<\kappa$ is a limit cardinal in $M$ then there is some limit ordinal $i<\kappa$, such that $\delta=\kappa_i$. Next, note that the Easton product above $i$ is $\delta^{++}$-closed, so it does not add subsets to $\delta$ nor it collapses $\delta^{++}$; and the product up to $i$ is $\delta^{+}$-c.c., so it does not collapse $\delta^{++}$ either. 

Finally, the same holds for $\kappa$ itself, although in $M$ there is no well-ordering of $\mathcal P(\kappa)$, so we have to settle for the fact that $\kappa^{++}$ injects into $\mathcal P(\kappa)$ by the same arguments as above.
\end{proof}

\subsection{Sets of reals and Dependent Choice} In recent times, there is a renewed interest in many ``irregularity properties'' of sets of reals consistent with the failure of the axiom of choice already at that level. Namely, the existence of Luzin sets, Hamel bases, etc., in models where $\RR$ cannot be well-ordered. The natural question after each resolve is whether or not $\DC$ can be added. Perhaps unsurprisingly, the answer is almost always positive. Much of this work has been done in \cite{BCSWY:2018} by Brendle, Castiblanco, Schindler, Wu, and Yu. We will prove a simpler result of the same flavor, using simplified arguments. For simplicity, all the results in this part assume $V=L$.

Recall that a \textit{Luzin set} is an uncountable set of reals whose intersection with every nowhere dense set is countable. It is a classic theorem that the Continuum Hypothesis implies the existence of a Luzin set, as well as forcing with $\Add(\omega,\omega_1)$ adds a Luzin set.\footnote{Or more generally, $\Add(\omega,2^{\aleph_0})$.}

Taking $\PP=\Add(\omega,\omega_1)$ with the permutation group of $\omega_1$ acting on $\PP$ by $\pi p(\pi\alpha,n)=p(\alpha,n)$, and $\sF$ is generated by $\fix(\alpha)$ for $\alpha<\omega_1$, where $\fix(\alpha)$ is $\{\pi\mid\pi\restriction\alpha=\id\}$. This symmetric system satisfies \autoref{lemma:chain-condition}, and therefore $\DC$ holds in the extension. 

Moreover, by a standard argument, the set of Cohen generics $A$ is in the model, but its enumeration is not. In particular, $\RR$ cannot be well-ordered there. Finally, $A$ is of course uncountable. And given any nowhere dense set, $F$, let $x$ be a code for $F$,\footnote{Since $\DC$ holds every Borel set has a code.} by c.c.c.\ there is a countable part of $\PP$ where $x$ was added, but then any $a\in A$ outside that part is Cohen generic over $L[x]$, and is therefore not in $F$. So $A\cap F$ is countable.

Replacing the Cohen reals by Sacks reals, and the finite support product by a countable support product, we lose the c.c.c.\ property, but we the forcing is still proper, as shown by Baumgartner in \cite{Baumgartner:1985}. By \autoref{lemma:proper} is enough to obtain $\DC$.\footnote{One can also note that the product of $\aleph_2$ copies of Sacks reals, over a model of $\CH$, will satisfy $\aleph_2$-c.c., so by taking the suitable construction just \autoref{lemma:chain-condition} provides us with $\DC_{\omega_1}$.} In this model we also have that every real was added by a countable part of the product, although in this case this is due to homogeneity rather than chain condition. In \cite{BCSWY:2018}, the construction goes on to force a Burstein set, which is a Hamel basis with an addition property of being a Bernstein set. This second forcing is $\sigma$-closed, so it preserves $\DC$.

This last part raises an interesting question. In \cite{BSWY:2018} the authors show that in Cohen's model there is a Hamel basis for $\RR$ over $\QQ$. Cohen's model is famous of having a Dedekind-finite set of reals, and therefore $\DC$ fails quite badly. However, it is also very different from Feferman's construction of a model satisfying $V=L(\RR)$ where the Boolean Prime Ideal theorem fails, in that the set of Cohen reals is in Cohen's model but not in Feferman's model. This is important because the proof in \cite{BSWY:2018} relies on this very fact. In \cite{BCSWY:2018} the construction goes through $L(\RR)$, where the set of Sacks reals is not present.

\begin{question}
Let $M$ be the symmetric extension obtained by forcing with a countable support product of Sacks reals of length $\omega_1$ as described above. Is there a Hamel basis for $\RR$ over $\QQ$ in $M$?
\end{question}

\subsection{Generic structures}
Wilfrid Hodges' influential paper \cite{Hodges:1974} presents six constructions of rings that have seemingly impossible properties, proving once more the necessity of the axiom of choice in the study of algebraic structures. His constructions rely on Lemma~3 called ``Removal of subsets'' in the paper which allows the transfer of a countable structure with certain properties to a model of $\ZF$ where the structure has only ``a few subsets''. The lemma then proved in \cite[Lemma~3.7]{Hodges:1976}. The proof goes through a more general construction, and then focuses on the case where $\kappa=\omega$, however by replacing $\omega$ by $\kappa$ (and finite by ${<}\kappa$) in the definitions relevant for the Removal of subsets, one immediately gets the consistency of $\DC_{<\kappa}$ with the modified lemma.

We extend this type of lemma to allow for $\DC_{<\kappa}$ to hold, if one assumes a little bit more. For the remainder of this section, $\cL$ is a fixed first-order language, and $\kappa$ is a fixed regular cardinal. 

For a $\cL$-structure $M$, we say that $X\subseteq M^n$ is \textit{$\kappa$-supported} if there exists $Y\subseteq M$ such that $|Y|<\kappa$, and $\pi$ is any automorphism which fixes $Y$ pointwise, then $X=\{\pi\vec x\mid\vec x\in X\}$. Similarly, a sequence of relations is $\kappa$-supported if it is uniformly $\kappa$-supported.

Finally, we say that $M$ is \textit{$\kappa$-homogeneous} if whenever $A\subseteq M$ and $|A|<\kappa$, if $B\subseteq M$ such that $f\colon A\to B$ is an isomorphism as $\cL$-substructures of $M$, then $f$ can be extended to an automorphism of $M$. It is well-known that if $M$ is $\kappa$-homogeneous, $A\equiv_N B$ for some $N\in[M]^{<\kappa}$, then there is an automorphism mapping $A$ to $B$ which fixes $N$ pointwise.

\begin{theorem}
Suppose that $M$ is a $\kappa$-homogeneous $\cL$-structure. There exists a symmetric extension $W\subseteq V[G]$ in which there is an $\cL$-structure such that $A\cong M$ in $V[G]$, but in $W$ the only subsets of $A$ are those which are $\kappa$-supported.
\end{theorem}
\begin{proof}
Let $\PP=\Add(\kappa,M\times\kappa)$, we define $\sG$ to be $\aut(M)\wr S_\kappa$, namely the wreath product of the automorphism group of $M$ with the permutation group of $\kappa$, which is itself a permutation group of $M\times\kappa$. A permutation $\pi\in\sG$ is made from an automorphism $\pi^*\in\aut(M)$, and for each $m\in M$ a permutation of $\kappa$, denoted by $\pi_m$, and $\pi(m,\alpha)=(\pi^*(m),\pi_m(\alpha))$. We define the action of $\sG$ on $\PP$ in the standard way, \[\pi p(\pi^*(m),\pi_m(\alpha),\beta)=p(m,\alpha,\beta).\]
Finally, for $N\subseteq M$ and $E\subseteq\kappa$ we define \[\fix(N,E)=\{\pi\in\sG\mid \pi^*\restriction N=\id\land\forall n\in N:\pi_n\restriction E=\id\},\] and $\sF$ is filter generated by $\{\fix(N,E)\mid N\in[M]^{<\kappa}, E\in[\kappa]^{<\kappa}\}$.

Indeed, it is not hard to see that the conditions of \autoref{lemma:closure} and $\DC_{<\kappa}$ must hold in the symmetric extension given by this symmetric system. Let $G$ be a $V$-generic filter for $\PP$, and let $W$ denote the symmetric extension.

For $m\in M$ and $\alpha<\kappa$ let $\dot x_{m,\alpha}$ be the name $\{\tup{p,\check\beta}\mid p(m,\alpha,\beta)=1\}$, $\dot a_m$ is the name $\{\dot x_{m,\alpha}\mid \alpha<\kappa\}^\bullet$ and $\dot A=\{\dot a_m\mid m\in M\}^\bullet$. Standard arguments show that for all $\pi\in\sG$, $\pi\dot x_{m,\alpha}=\dot x_{\pi^*(m),\pi_m(\alpha)}$ and $\pi\dot a_m=\dot a_{\pi^*(m)}$, and so $\pi\dot A=\dot A$. Therefore all these names are symmetric.

Moreover, since the $M$-part of $\pi\in\sG$ is an automorphism, if $R$ is a symbol in $\cL$, then $\{\dot a_{\vec m}\mid\vec m\in R^M\}^\bullet$ is symmetric, where $\dot a_{\vec m}=\tup{\dot a_{m_i}\mid \vec m=\tup{m_i\mid i<\alpha}}^\bullet$. In particular in $W$ there is a natural way of interpreting $A$ as an $\cL$-structure, and clearly in $V[G]$ it holds that $A\cong M$ by $m\mapsto a_m$.

It remains to show that if $B\subseteq A$ is in $W$, then $B$ is $[M]^{<\kappa}$-supported. Let $\dot B$ be a name for $B$ in $\HS$ and $\fix(N,E)\subseteq\sym(\dot B)$. If $p\forces``\dot B\text{ is not }\kappa\text{-suppported}"$, then in particular $N$ itself is not a support for $B$, then there is an automorphism $\pi^*$ which fixes $N$ pointwise and moves an element $B$ outside of $B$ itself. The problem is that this automorphism might be generic. 

However, let $a_m\in B$ and $a_{m'}\notin B$ such that there is such $\sigma^*(a_m)=a_{m'}$. In particular $m$ and $m'$ have the same type over $N$. Since this statement is absolute to $V$, we can therefore assume without loss of generality that $\sigma^*\in V$, and therefore there is a suitable $\pi\in\sG$ for which $\pi^*=\sigma^*$.

Moreover, we can assume that $\pi_a$ and $\pi_b$ are such that $\pi p$ is compatible with $p$, simply by ensuring the domains on the $a$ and $b$ coordinates of $p$ become disjoint. Therefore, $\pi p\forces\dot a_{m'}\in\dot B$, but since $p$ and $\pi p$ are compatible, this is impossible.
\end{proof}

We draw some easy corollaries. The first is that $\kappa$-amorphous sets are consistent with $\DC_{<\kappa}$, where a set is $\kappa$-amorphous if it cannot be written as a union of two subsets neither of which is of size ${<}\kappa$.

\begin{corollary}
It is consistent with $\DC_{<\kappa}$ that there exists a set whose cardinality is not ${<}\kappa$, but every subset is either of size ${<}\kappa$ or its complement is of size ${<}\kappa$.
\end{corollary}
The next corollary was proved by the author in \cite{Karagila:2012}.
\begin{corollary}
It is consistent with $\DC_{<\kappa}$ that there is a vector space over any fixed field which is not generated by ${<}\kappa$ vectors, but any proper subspace has dimension ${<}\kappa$.
\end{corollary}
Taking a countable field and $\kappa=\omega_1$ we obtain the following corollary.
\begin{corollary}
It is consistent with $\DC$ that there is an uncountable Abelian group such that all of its proper subgroups are countable.
\end{corollary}

\begin{remark}The reason we used $\Add(\kappa,M\times\kappa)$ and not $\Add(\kappa,M)$ is that we needed to create a better set-theoretic indiscernibility between the $a_m$'s. If one repeats the proof using only $\Add(\kappa,M)$, then one discovers that sets such as $\{a_m\mid 0\in a_m\}$ enter the model, and they have nothing to do with being supported. However, doing that does offer one advantage of obtaining failures as subsets of the reals. So for example, one could $\Add(\omega,M)$ or use a countable support product of Sacks reals, and obtain the generic structure as a structure on a set of reals. 
\end{remark}
\providecommand{\bysame}{\leavevmode\hbox to3em{\hrulefill}\thinspace}
\providecommand{\MR}{\relax\ifhmode\unskip\space\fi MR }
\providecommand{\MRhref}[2]{%
  \href{http://www.ams.org/mathscinet-getitem?mr=#1}{#2}
}
\providecommand{\href}[2]{#2}

\end{document}